\documentclass{amsart}

\usepackage[inactive]{srcltx} % SRC Specials for DVI Searching
\usepackage{amsmath, amsthm, amscd, amsfonts, amssymb, graphicx, color, extarrows}
 \newtheorem{thm}{Theorem}[section]
 \newtheorem{cor}[thm]{Corollary}
 \newtheorem{lem}[thm]{Lemma}

 \theoremstyle{definition}
 
 \theoremstyle{remark}

  \newtheorem{exm}[thm]{Example}
\newcommand{\A}{\mathcal{A}}

\newcommand{\al}{Alg\mathcal{N}}
\newcommand{\alf}{Alg_{\textbf{F}}\mathcal{N}}
\newcommand{\bo}{\textbf{B}(\mathcal{X})}
\begin{document}

\title[Ternary derivations of nest algebras]
 {Ternary derivations of nest algebras}

\author{Hoger Ghahramani}
\thanks{{\scriptsize
\hskip -0.4 true cm \emph{MSC(2010)}:  47L35, 47B47, 16W25.
}}

\address{Department of
Mathematics, University of Kurdistan, P. O. Box 416, Sanandaj,
Iran.}

\email{h.ghahramani@uok.ac.ir; hoger.ghahramani@yahoo.com}

\address{}

\email{}

\thanks{}

\thanks{}



\date{}

\dedicatory{}

\commby{}

%%% ----------------------------------------------------------------------

\begin{abstract}
Suppose that $ \mathcal{X} $ is a (real or complex) Banach space, $dim \mathcal{X} \geq 2$, and $\mathcal{N}$ is a nest on $\mathcal{X}$, with each $N \in \mathcal{N}$ is complemented in $\mathcal{X}$ whenever $N_{-}=N$. A ternary derivation of $\al$ is a triple of linear maps $(\gamma, \delta, \tau)$ of $\al$ such that $\gamma(AB)=\delta(A)B+A\tau(B)$ for all $A,B\in \al$. We show that for linear maps $\delta, \tau$ on $\al$ there exists a unique linear map $\gamma:\al \rightarrow \al$ defined by $\gamma(A)=RA+AT$ for some $R,T \in \al$ such that $(\gamma, \delta, \tau)$ is a ternary derivation of $\al$ if and only if $\delta , \tau$ satisfy $\delta(A)B+A\tau(B)=0$ for any $A,B\in \al$ with $AB=0$. We also prove that every ternary derivation on $\al$ is an inner ternary derivation. Our results are applied to characterize the (right or left) centralizers and derivations through zero products, local right (left) centralizers, right (left) ideal preserving maps and local derivations on nest algebras.
\end{abstract}

%%% ----------------------------------------------------------------------
\maketitle
%%% ----------------------------------------------------------------------

\section{Introduction}
Let $\mathcal{A}$ be an (associative) algebra. A \textit{ternary derivation} of $\mathcal{A}$ is a triple of linear maps $(\gamma, \delta, \tau)$ of $\A$ such that 
\[ \gamma(ab)=\delta(a)b+a\tau(b)\]
for all $a,b\in \mathcal{A}$. The set of all ternary derivations of $\mathcal{A}$ is denoted
by $Tder(\mathcal{A})$. The notion of ternary derivations generalizes several classes of linear mappings; for example, if $\gamma = \delta = \tau$, then $\gamma$ is a derivation of $\mathcal{A}$, and if $\gamma = \delta$, then $\gamma$ is a generalized derivation of $\mathcal{A}$. If $\A$ is unital with unity $1$, routine verifications show that $\gamma:\A \rightarrow \A$ is a generalized derivation if and only if $\gamma(ab)=\gamma(a)b+a\gamma(b)-a\gamma(1)b$ for all $a,b\in \A$. The notion of ternary derivation was introduced by Jimen$\acute{\text{e}}$z-Gestal and P$\acute{\text{e}}$rez-Izquierdo in \cite{jim}. They characterized ternary derivations of the generalized Cayley-Dickson algebras over a field of characteristic not 2 and 3. In \cite{jim2}, ternary derivations of finite-dimensional real division algebras were studied. More recently, ternary derivations of separable associative and Jordan algebras were described by Shestakov \cite{she}. Ternary derivatives are also defined and studied on nonassociative algebras \cite{she2, zh}. Refer to \cite{bar, pe} for a study of the motivation for defining ternary derivations and related content. 
\par 
In \cite{she}, the Inner ternary derivations is defined as follows: a ternary derivation $(\gamma, \delta, \tau)$ of an algebra $\A$ is called an \textit{inner ternary derivation} if there exist $a, b, c \in \A$ such that
\[ (\gamma, \delta, \tau) =(L_a + R_b, L_a + R_c, - L_c + R_b).\]
Recall that $L_a(b) := ab := R_b(a)$ for any $a,b\in \A$, are the \textit{left} and \textit{right multiplication operators}. Clearly, each inner ternary derivation is a ternary derivation. The converse is, in general, not true (see \cite{she}). Among the interesting problems in the theory of derivations is the identification of rings whose derivations are inner. Many studies have been performed in this regard and it has a long history. We may refer to \cite{ash} and the references therein for more information. So it seems reasonable to consider the problem of innerness of ternary derivations. In \cite{she}, innerness of ternary derivations was proved on some algebras. A characterization for a ternary derivation to be inner on triangular algebras is given in \cite{bar}. In this paper we show that any ternary derivation on a nest algebra is an inner ternary derivation.
\par 
Let $\mathcal{A}$ be an algebra and $(\gamma, \delta, \tau)\in Tder(\mathcal{A})$. Then $\delta , \tau$ satisfy
\[ a,b\in \mathcal{A}, \, \, ab=0 \Longrightarrow \delta(a)b+a\tau(b)=0. \quad (Z) \]
The following example shows that the converse of this observation as follows is not necessarily true.
\begin{exm}
Let $\mathbb{C}$ be the field of complex numbers. Consider the $\mathbb{C}$-algebra $\A$ of the form
\[\A :=\Bigg\lbrace\begin{pmatrix}
 a & b & c\\
  0 & a & d\\
  0 & 0 & a
\end{pmatrix} \, : \, a,b,c,d \in\mathbb{C} \Bigg\rbrace \]
under the usual matrix operations. $\A$ is a unital algebra with the identity matrix $I$. Let $X=\begin{pmatrix}
 0 & 0 & 0\\
  0 & 0 & 1\\
  0 & 0 & 0
\end{pmatrix} $ and define $\mathbb{C}$-linear maps $\delta: \A \rightarrow \A$ and $\tau: \A \rightarrow \A$ by $\delta(A):=R_{X}(A)=AX$ and $\tau(A):= L_{X}(A)=XA$. If $AB=0$, where $A=\begin{pmatrix}
 a & b & c\\
  0 & a & d\\
  0 & 0 & a
\end{pmatrix} $ and $B=\begin{pmatrix}
 a^{\prime} & b^{\prime} & c^{\prime}\\
  0 & a^{\prime} & d^{\prime}\\
  0 & 0 & a^{\prime}
\end{pmatrix} $, then $aa^{\prime}=0$ and $ab^{\prime}+ ba^{\prime}=0$. So $ba^{\prime}=0$ and $AXB=0$. Hence 
\[ \delta(A)B+A\tau(B)=0.\]
Suppose that there is a $\mathbb{C}$-linear map $\gamma : \A \rightarrow \A$ such that $(\gamma, \delta, \tau)\in Tder(\mathcal{A})$. Therefore, $\gamma(AB)=2AXB$ for all $A,B \in \A$. If $A=I$, then $\gamma(B)=2XB$ for all $B\in \A$ and if $B=I$, then $\gamma(A)=2AX$ for all $A\in \A$. Hence $AX=XA$ for all $A\in \A$. But $X$ is not belong to $Z(\A)$ (the centre of $\A$). This contradiction shows that the linear maps $\delta, \tau$ on $\A$ satisfy $(Z)$, but there is not a linear map $\gamma$ on $\A$ such that $(\gamma, \delta, \tau)\in Tder(\mathcal{A})$.
\end{exm}
Given this example and the previous discussion, the following question naturally arises.
\par 
\textbf{Question} Which algebra $ \A $ has the following property: for given linear maps $\delta , \tau$ on $\A$ satisfying $(Z)$, there exists a linear map $\gamma$ on $\A$ such that $(\gamma, \delta, \tau)\in Tder(\mathcal{A})$?
\\
\par 
One of the interesting issues in mathematics is the determination of the structure of linear maps on algebras that act through zero products in the same way as certain mappings, such as homomorphisms, derivations, centralizers, etc. (see, for instance, \cite{bre1, fad, gh1, gh2, gh3, qi} and the references therein). Among these issues, one can point out the problem of characterizing linear maps $\delta, \tau$ on an algebra $\A$ which satisfy $(Z)$, (\cite{barar, ben, lee}, among others). Given this characterization, it is possible to obtain an answer to the question. In \cite[Theorem 2.1]{barar}, it was proved that if $\A$ is a unital standard operator algebra on a complex Banach space $\mathcal{X}$ with $dim \mathcal{X} \geq 2$ and $\delta, \tau : \A \rightarrow \bo$ are linear maps satisfying $(Z)$, then there exist $ R , S, T  \in \bo$ such that $ \delta (A) = A S   - R A$, $\tau ( A) = A T - S A$ for all $ A \in \mathcal{A} $. According to this result, one can find a positive answer to the question on standard operator algebras. In addition, it can be shown that on standard operator algebras with the above conditions every ternary derivation is an inner ternary derivation. Using \cite[Proposition 2.12]{gh0} and its proof, it can be concluded that the answer to the question is correct for zero product determined algebras. In general, we do not know that a nest algebra is a zero product determined algebra. In this paper, by applying operator theory methods, we show that the answer to the question is positive for nest algebras on Banach spaces under some suitable assumption. In fact, the following theorem is the main result of the article.
\begin{thm}\label{main}
Let $\mathcal{X}$ be a (real or complex) Banach space with $dim \mathcal{X} \geq 2$, let $\mathcal{N}$ be a nest on $\mathcal{X}$, with each $N\in \mathcal{N} $ is complemented in $\mathcal{X}$ whenever $N_{-}=N$, and let $\delta , \tau : \al \rightarrow \al$ be linear maps. Then $\delta, \tau $ satisfy $(Z)$ if and only if there exits a unique linear map $\gamma : \al \rightarrow \al$ defined by $\gamma(A)=RA+AT$ for some $R,T \in \al$ such that $(\gamma, \delta, \tau)\in Tder(\al)$.
\end{thm}
In addition to answering the question, this theorem also gives us the innernes of ternary derivations. We have the following corollary.
\begin{cor}\label{inner}
Let $\mathcal{X}$ be a (real or complex) Banach space with $dim \mathcal{X} \geq 2$, let $\mathcal{N}$ be a nest on $\mathcal{X}$, with each $N\in \mathcal{N} $ is complemented in $\mathcal{X}$ whenever $N_{-}=N$. Then every ternary derivation on $\al$ is an inner ternary derivation.
\end{cor}
We observe that the nests on Hilbert spaces, finite nests and the nests having order-type $ \omega+1$ or $1+\omega^{*}$, where $\omega$ is the order-type of the natural numbers, satisfy the condition in Theorem~\ref{main} and Corollary~\ref{inner} automatically.
\par 
As a corollary of our main results we characterize linear maps $\delta,\tau$ on $\al$ satisfying $(Z)$. We also present various applications of Theorem~\ref{main} for determining (right or left) centralizers and derivations through zero products, local right (left) centralizers, right (left) ideal preserving maps and local derivations on nest algebras.
\par 
This article is organized as follows: In Section 2 we provide the definition of nest algebra and some of the required results. In Section 3, the applications of Theorem~\ref{main} are presented. Section 4 is devoted to the proof of Theorem~\ref{main} and Corollary~\ref{inner}.
% ------------------------------------------------------------------------_________________________________
%%_____________________________________________________________________________________
\section{preliminaries and tools}
Let $\mathcal{X}$ be a (real or complex) Banach space, let $\bo$ be the Banach algebra of all bounded linear operators on $\mathcal{X}$, and let $\textbf{F}(\mathcal{X})$ be the ideal of all finite rank operators in $\bo$. A \textit{nest} $\mathcal{N}$ on $\mathcal{X}$ is a chain of closed (under norm topology) subspaces of $\mathcal{X}$ with $\lbrace 0 \rbrace$ and $\mathcal{X}$ in $\mathcal{N}$ such that for every family $\lbrace N_{i}\rbrace$ of elements of $\mathcal{N}$, both $\bigcap N_{i}$ and $\bigvee N_{i}$ (closed linear span of $\lbrace N_{i} \rbrace$) belong to $\mathcal{N}$. The \textit{nest algebra} associated to the nest $\mathcal{N}$, denoted by $\al$, is the weak closed operator algebra of all operators in $\bo$ that leave members of $\mathcal{N}$ invariant. We say that $\mathcal{N}$ is \textit{non-trivial} whenever $\mathcal{N} \neq \lbrace \lbrace 0\rbrace, \mathcal{X}\rbrace$. The ideal $\al \bigcap \textbf{F}(\mathcal{X})$ of all finite rank operators in $\al$ is denoted by $\alf$ and for $N\in \mathcal{N}$,
\[ N_{-}:=\bigvee \lbrace M\in \mathcal{N} \, \vert \, M\subset N \rbrace.\]
The identity element of algebras denotes by $I$ and an element $P$ in an algebra is called an idempotent if $P^{2}=P$. In order to prove our results we need the following results.
\begin{lem}$($\cite[Lemma 3.2]{hou}$)$\label{p1}
Let $\mathcal{N}$ be a nest on a Banach space $\mathcal{X}$. If $N\in \mathcal{N}$ is complemented in $\mathcal{X}$ whenever $N_{-}=N$, then the ideal $\alf$ of finite rank operators of $\al$ is contained in the linear span of the idempotents in $\al$. 
\end{lem}
%%_______________________________________________________%______________
\begin{lem}\label{p2}
Let $\mathcal{N}$ be a nest on a Banach space $\mathcal{X}$, and $N\in \mathcal{N}$ be complemented in $\mathcal{X}$ whenever $N_{-}=N$. Then 
\begin{itemize}
\item[(i)] 
$\big\lbrace T\in \bo \, \vert \, TF=0 \, \text{for all} \,\, F\in \alf \rbrace =\lbrace 0 \big\rbrace$,
\item[(ii)] $\big\lbrace T\in \bo \, \vert \, FT=0 \, \text{for all} \,\, F\in \alf \rbrace =\lbrace 0 \big\rbrace$.
\end{itemize}
\end{lem}
\begin{proof}
(i) Suppose that $T\in \bo$ and $TF=0$ for all $F\in \alf$. By \cite{sp} we have $\overline{\alf}^{SOT}=\al$. Thus there exists a net $(F_{\gamma})_{\gamma \in \Gamma}$ in $\alf$ converges to the identity operator $I$ with respect to the strong operator topology. Since the product of $\bo$ is separately SOT-continuous, it follows that $TF_{\gamma}\longrightarrow T$ in the strong operator topology. So $T=0$.
\par 
(ii) The proof is obtained by using a similar argument as in (i).
\end{proof}
Let $\mathcal{N}$ be a nest on the Banach space $\mathcal{X}$. If $\mathcal{X}$ is a Hilbert space or $\mathcal{N}$ is a finite nest or a nest having order-type $ \omega+1$ or $1+\omega^{*}$, where $\omega$ is the order-type of the natural numbers, then it is obvious that $N\in \mathcal{N}$ is complemented in $\mathcal{X}$ whenever $N_{-}=N$.
%%___________________________________________________________________________
%%_________________________________________________________________________________
\section{Applications}
In this section, we present applications of Theorem~\ref{main}. From this point up to the last section $ \mathcal{X} $ is a (real or complex) Banach space, $dim \mathcal{X} \geq 2$, and $\mathcal{N}$ is a nest on $\mathcal{X}$, with each $N \in \mathcal{N}$ is complemented in $\mathcal{X}$ whenever $N_{-}=N$.
%%********************************************************************************************************
\subsection*{Characterizing linear maps $\delta, \tau$ satisfying $(Z)$}
Determining the structure of linear maps satisfying $(Z)$ is a matter of interest, as explained in the introduction. In the following corollary we characterize linear maps $\delta,\tau$ on $\al$ satisfying $(Z)$.
%%______________________________________________________________________________
\begin{cor}\label{charac}
Let $ \delta, \tau: \al \rightarrow \al $ be linear maps. Then $\delta, \tau$ satisfy $(Z)$ if and only if there exist $R,S,T\in \al$ such that $\delta(A)=RA+AS$ and $\tau(A)=-SA+AT$ for all $A\in \al$.
\end{cor}
\begin{proof}
Suppose that $\delta, \tau$ satisfy $(Z)$. By Theorem~\ref{main} there exists a linear map $\gamma: \al \rightarrow \al$ such that $(\gamma, \delta, \tau)\in Tder(\al)$ and by Corollary~\ref{inner}, $(\gamma, \delta, \tau)$ is an inner ternary derivation. So there exist $R,S,T\in \al$ such that $\delta(A)=RA+AS$ and $\tau(A)=-SA+AT$ for all $A\in \al$. The converse is clear.
\end{proof}
%%_________________________________________________________________________
If in $(Z)$ we assume that $\tau=\delta$, then $\delta$ is like derivation at zero product elements. Determining the structure of linear maps behaving like derivations at zero product elements is a topic of interest that has been studied extensively. See for instance \cite{barar, ben, bre1, gh1, gh2, lee, qi} and the references therein. From Theorem~\ref{main}, one gets the following corollary. %%%_________________________________________________________
\begin{cor} \label{der}
Assume that $ \delta : \al \to \al$ is a linear map. Then $\delta$ satisfying
\[ A B = 0 \Longrightarrow A \delta ( B) + \delta ( A ) B = 0 \quad (A, B \in \al) .  \]
if and only if there exist $ S,T \in \al $ such that $ \delta (A) = A T   - S A  $ for all $ A \in \mathcal{A} $ and $ T - S \in Z ( \al) $.  
\end{cor}
\begin{proof}
By Corollary~\ref{charac}, there exist $ R , S , T \in B ( \mathcal{X} ) $ such that
\[ \delta ( A)=RA+AS = -SA+AT\]
for all $ A \in \mathcal{A} $. Hence, $ A ( T-S ) = ( R + S ) A $ for all $ A \in \al $. Let $ A = I $, we see that $ T-S = R + S $. Therefore, $ T-S \in Z(\al) $. The converse is clear.
\end{proof}
The above corollary is obtained in some previous articles. For example, this conclusion can be drawn from \cite[Proposition 2.3]{ghJor} and \cite[Theorem 4.1]{ghJor}. So Corollary~\ref{charac} is a generalization of this result.
%%**********************************************************************************************
\subsection*{Characterizing (right or left) centralizers through zero products}
Let $\mathcal{A}$ be an algebra. A linear map $\rho : \mathcal{A} \rightarrow \mathcal{A}$ is called a \textit{right (left) centralizer} if $\rho(ab) = a\rho(b)$ ($\rho(ab) = \rho(a)b$) for each $a,b \in \mathcal{A}$ and $\rho$ is called a \textit{centralizer} if it is both a left centralizer and a right centralizer. One of the issues to consider is to describe the structure of linear maps that act as right (left) centralizers or centralizers at zero product elements as follows:
\begin{equation}\label{d2}
\begin{split}
& ab = 0 \Longrightarrow a \rho ( b) = 0 ,\\
& ab = 0 \Longrightarrow  \rho ( a)b = 0 , \\
& ab = 0 \Longrightarrow  a \rho ( b) =\rho ( a)b = 0 ,
\end{split}
\end{equation}
where $a,b \in \mathcal{A}$ and $\rho : \mathcal{A} \rightarrow \mathcal{َA}$ is a linear map. In \cite{bre1}, Bre$\check{\textrm{s}}$ar shows that if $\mathcal{A}$ is a prime ring and $\rho$ is an additive map on $\mathcal{A}$, then $\rho$ satisfying the second equation in \eqref{d2} if and only if $\rho$ is a left centralizer. This problem has been explored by several authors, (\cite{gh11, lli, li, qi, xu} among others). Now we consider this problem on nest algebras.
%%________________________________________________________________________
\begin{cor}\label{3main}
Assume that $ \tau, \delta, \rho : \al \to \al $ are linear maps.
\begin{itemize}
\item[(i)] $\tau$ satisfies 
\[ A B = 0 \Longrightarrow A \tau ( B) = 0  \quad (A, B \in \al)  \]
if and only if $ \tau ( A) = A D $ for all $ A\in \al $ in which $ D \in \al  $.
\item[(ii)] $\delta$ satisfies 
\[ A B = 0 \Longrightarrow  \delta ( A)B = 0 \quad (A, B \in \al)  \]
if and only if $ \delta ( A) = DA $ for all $ A\in \al $ in which $ D \in \al  $.
\item[(iii)] $\rho$ satisfies 
\[ A B = 0 \Longrightarrow  A \rho ( B) =\rho ( A)B = 0 \quad (A, B \in \al)  \]
if and only if $ \rho ( A) = AD $ for all $ A \in \al $ in which $ D \in \al $ and $AD=DA$ for all $ A\in \al $. 
\end{itemize}
\end{cor}
\begin{proof}
(i) Let $A \tau ( B) = 0$ whenever $AB=0$. Set $\delta =0$. Then $\delta , \tau$ satisfy $(Z)$. By Theorem~\ref{main} there exists a linear map $\gamma: \al \rightarrow \al$ such that $(\gamma, \delta, \tau)\in Tder(\al)$. So $\gamma(AB)=A\tau(B)$ for all $A,B\in \al$. Hence $\gamma=\tau$ and $\tau(AB)=A\tau(B)$ for all $A,B\in \al$. By setting $D=\tau(I)\in \al$ we see that $ \tau ( A) = A D $ for all $ A\in \al $. The converse is clear.
\par
(ii) The proof is obtained by using a similar argument as in (i).
\par
(iii) It is clear from (i) and (ii).
\end{proof}
%%**********************************************************************************************
\subsection*{Local right (left) centralizers and right (left) ideal preserving maps}
A linear map $\psi$ on an algebra $\mathcal{A}$ is called a \textit{local right (left) centralizer} if for any $a\in \A$ there exists a right (left) centralizer $\rho_{a}:\A\rightarrow \A$ (depending on $a$) such that $\psi(a)=\rho_{a}(a)$. Clearly, each right (left) centralizer is a local right (left) centralizer. The converse is, in general, not true. We say that a linear map $\psi:\A\rightarrow \A$ is \textit{right (left) ideal preserving} if $\psi(\mathcal{J})\subseteq \mathcal{J}$ for any right (left) ideal $\mathcal{J}$ of $\A$. Suppose that $\A$ is a unital algebra. It is then easily verified that the linear map $\psi:\A\rightarrow \A$ is right (left) ideal preserving if and only if $\psi$ is a local right (left) centralizer. So it is clear that any right (left) centralizer is a right (left) ideal preserving map, but the converse is not necessarily true. It is natural and interesting to ask for what algebras any local right (left) centralizer or any right (left) ideal preserving map is a right (left) centralizer. Johnson \cite{john} has proven that if $\A$ is a semisimple Banach algebra with an approximate identity and $\psi$ is a bounded operator on $\A$ that leaves invariant all closed left ideals of $\A$, then $\psi$ is a left centralizer of $\A$. Hadwin and Li \cite{had2} have shown that Johnson's Theorem holds for all CSL algebras. In particular Hadwin, Li and their collaborators \cite{had3, had, had4, li} have studied the problems of this kind in the past twenty years for various reflexive operator algebras. This problem has also been investigated for other algebras. We refer to \cite{gh4, ka} and their references.
%%__________________________________________________________________
In the next corollary we characterize the local right centralizers and local left centralizers of the nest algebras. 
\begin{cor} \label{4main}
Let $\psi:\al\rightarrow \al$ be a linear map.
\begin{itemize}
\item[(i)] $\psi$ is a local right centralizer if and only if $\psi$ is a right centralizer.
\item[(ii)] $\psi$ is a local left centralizer if and only if $\psi$ is a left centralizer.
\end{itemize}
\end{cor} 
\begin{proof}
(i) Suppose that $\psi$ is a local right centralizer. Therefore, for any $A\in\al$, there is an element $D_{A}\in\al$ such that $\psi(A)=AD_A$. So for $A,B \in \al$ with $AB=0$, we have 
\[A\psi(B)=ABD_B=0.\]
From Corollary~\ref{3main}(i), it follows that $\psi$ is a a right centralizer. The converse is clear.
\par 
(ii) By using Corollary~\ref{3main}(ii) and a similar proof as (i), we obtain the desired conclusion.
\end{proof}
%__________________________________________________________________
The right ideal preserving linear maps and left ideal preserving linear maps on nest algebras are described in the following corollary.
\begin{cor} \label{5main}
Let $\psi:\al\rightarrow \al$ be a linear map.
\begin{itemize}
\item[(i)] $\psi$ is a right ideal preserving map if and only if $\psi$ is a right centralizer.
\item[(ii)] $\psi$ is a left ideal preserving map if and only if $\psi$ is a left centralizer.
\end{itemize}
\end{cor} 
\begin{proof}
(i) Assume that $\psi$ is a right ideal preserving map. Let $A\in \al$. It is clear that $A(\al)$ is a right ideal of $\al$. It follows from the hypothesis that $\psi(A(\al))\subseteq A(\al)$. By the fact that $\al$ is unital, there exists an element $D_A\in \al$ such that $\psi(A)=AD_A$. So $\psi$ is a local right centralizer and by Corollary~\ref{4main}(i), it is a right centralizer. Conversely, if $\psi$ is a right centralizer, a routine verification shows that $\psi$ is a right ideal preserving map.
\par 
(ii) By using Corollary~\ref{4main}(ii) and a similar proof as (i), we obtain the desired conclusion.
\end{proof}
%%**********************************************************************************************
\subsection*{Local (generalized) derivations}
A linear map $\delta$ on an algebra $\A$ is called a \textit{local (generalized) derivation} if for any $a\in \A$ there is a (generalized) derivation $\delta_{a}:\A\rightarrow \A$ (depending on $a$) such that $\delta(a)=\delta_{a}(a)$. There have been many papers in the literature investigating when local (generalized) derivations are (generalized) derivations, see \cite{bre1, de, had2, had4, ka, lli2, li, zha} and the references therein. In the following corollary, we characterize local generalized derivations on nest algebras.
\begin{cor}\label{localgd}
For any linear map $\delta:\al \rightarrow \al$,The following are equivalent:
\begin{itemize}
\item[(i)] $\delta$ is a generalized derivation;
\item[(ii)] $\delta$ is a local generalized derivation;
\item[(iii)] $A\delta(B)C=0$, whenever $A,B,C\in \al$ such that $AB=BC=0$.
\end{itemize}
\end{cor}
\begin{proof}
According to the Corollary~\ref{3main}, all the conditions of the \cite[Proposition 1.1]{li} are satisfied, so we obtain the desired result from \cite[Proposition 1.1]{li}.
\end{proof}
We have the following result for the local derivations.
\begin{cor}\label{locald}
Let  $\delta:\al \rightarrow \al$ be a linear map. Then $\delta$ is a local derivation if and only if $\delta$ is a derivation.
\end{cor}
\begin{proof}
If $\delta:\al \rightarrow \al$ is a local derivation, then it is a local generalized derivation and $\delta(I)=0$. By Corollary~\ref{localgd}, $\delta$ is a generalized derivation with $\delta(I)=0$. So $\delta$ is a derivation. The converse is clear.
\end{proof}
It should be noted that the condition on the nests in Corollary~\ref{localgd} is different from the conditions on the nests in \cite{de}. Corollary~\ref{localgd} also generalizes similar results for nest algebras on Hilbert spaces.
%%___________________________________________________________________________
%%_________________________________________________________________________________
\section{Proofs}
\textbf{Proof of Theorem~\ref{main}:} Only the 'only if' part needs to be checked. Let $\delta , \tau : \al \rightarrow \al$ be linear maps satisfy $(Z)$. Through the following steps we prove that there exits a linear map $\gamma : \al \rightarrow \al$ defined by $\gamma(A)=RA+AT$ for some $R,T \in \al$ such that $(\gamma, \delta, \tau)\in Tder(\al)$.
\\
\\
\textbf{Step 1.} For each $A\in \al$ and each idempotent element $P\in \al$, we have
\[ \delta ( A P )+ AP \tau (I) =A \tau ( P ) + \delta ( A) P\quad \text{and} \quad \tau ( P A ) + \delta ( I) P A = P \tau (A) + \delta ( P ) A . \]
 \par 
 Let $A, P\in \al$ where $P^2=P$. The operator $ I - P $ is an idempotent and $ AP(I-P) = 0 $. By assumption we have
 \[  \delta ( AP ) (I-P) +A P \tau ( I-P )  = 0 . \]
Hence
\[  \delta ( AP ) - \delta ( AP ) P+AP \tau (I) - AP \tau (P)  = 0 . \]
Therefore,
\begin{equation} \label{eq3}%(1)
 \delta ( AP )+AP \tau (I)  =   \delta (AP) P+AP \tau (P) .
\end{equation}
Since $ A(I-P) P = 0 $, it follows that
\[ \delta ( A(I-P) ) P+A (I-P)  \tau (P) = 0 .\]
So
\[\delta (A) P - \delta ( AP ) P+ A \tau (P) - A P \tau (P) = 0 . \]
Consequently 
\begin{equation} \label{eq4}%(2)
  \delta ( AP ) P+A P \tau (P) =\delta ( A)P+A \tau (P) . 
\end{equation}
By comparing \eqref{eq3} and \eqref{eq4}, we arrive at
\[  \delta  (AP) +AP \tau (I)= \delta (A) P+A \tau (P) . \]
Since $ P(I-P) A = 0 $ and $ (I-P) PA = 0$, by assumption we have 
\[ \delta (P) (I-P) A +P \tau ( (I-P) A )= 0 \quad \text{and} \quad  \delta ( I-P ) PA+ Q \tau (PA) = 0 . \]
From these equations we have the followings, respectively
\[   \delta (P) PA+P \tau ( PA )= \delta ( P ) A+P \tau ( A) \]
and
\[\delta (P) PA+P \tau (PA)=\delta (I) PA+\tau ( PA) . \]
Comparing these equations, we get
\[\delta (I) PA+ \tau ( PA) = \delta (P) A+P \tau (A) . \]
\\
%%-----------------------------------------------------------------------------------------------------------------------
\textbf{Step 2.} For each $A\in \al$ and $F\in \alf$ we have
\[ \delta ( A F )+ AF \tau (I) =A \tau ( F ) + \delta ( A) F\quad \text{and} \quad \tau ( F A ) + \delta ( I) F A = F \tau (A) + \delta ( F) A . \]
\par
By Step 1, Lemma~\ref{p1} and the fact that $\delta, \tau$ are linear we get the desired result.
\\ \\
%%-----------------------------------------------------------------------------------------------------------------------
\textbf{Step 3.} For all $ A,B \in \mathcal{A} $, we have
\[ \delta ( AB) = A \delta ( B) + \delta (A) B - A \delta (I) B . \]
\par
Taking $ A= I $ in Step 2, we find that 
\begin{equation} \label{eq5}
\delta (F) = \tau (F) - F \tau (I) + \delta (I ) F ,
\end{equation}
for all $ F \in \alf $. Since $  \alf $ is an ideal in $ \al $, it follows from \eqref{eq5} that
\[ \delta ( AF) = \tau ( AF ) - AF \tau (I) + \delta (I) AF  \]
for all $ A \in \al $ and $ F \in \alf $. From this equation and Step 2, we see that  
\begin{equation} \label{eq6}
\tau ( AF) = A \tau (F) + \delta (A) F - \delta (I) A F 
\end{equation}
for all $ A \in \al $ and $ F \in \alf$. From \eqref{eq6}, we get
\begin{equation} \label{eq7}
\tau ( AB F) = AB \tau (F) + \delta ( AB) F - \delta (I) AB F ,
\end{equation}
for all $ A , B \in \al $ and $ F \in \alf$. On the other hand, 
\begin{align} \label{eq8}
\tau ( ABF) & = A \tau (BF) + \delta (A) B F - \delta (I) AB F  \nonumber \\
& = AB \tau (F) + A \delta (B) F - A \delta (I) B F + \delta (A) B F - \delta (I) AB F 
\end{align}
for all $ A , B \in \al $ and $ F \in \alf $. By comparing \eqref{eq7} and \eqref{eq8}, we arrive at
\[ \delta (AB) F = A \delta (B) F + \delta (A) B F- A \delta (I) BF    \]
for all $ A , B \in \al$ and $ F \in \alf $. From Lemma~\ref{p2}(i), it follows that
\[ \delta ( AB ) = A \delta (B) + \delta (A) B - A \delta (I) B \]
for all $ A , B \in \al$.
\\ \\
%%-----------------------------------------------------------------------------------------------------------------------
\textbf{Step 4.} For all $ A,B \in \al $, we have
\[ \tau (AB) = A \tau (B) + \tau (A) B - A  \tau (I) B . \]
\par
From Step 2 and \eqref{eq5}, it follows that
\begin{align} \label{eq9}
\tau ( F A) & = F \tau (A) + \delta (F) A - \delta (I) F A \nonumber \\
& = F \tau ( A) + \tau (F) A - F \tau (I) A , 
\end{align}
for all $ A  \in \al $ and $ F \in \alf $. Now, by using \eqref{eq9} for all $ A , B \in \al $ and $ F \in \alf $, we calculate $ \tau ( F AB ) $ in two ways and we obtain the followings.
\[ \tau ( F AB ) = F \tau (AB) + \tau (F) AB - F \tau (I) AB \]
and
\[ \tau (FAB ) = FA \tau (B) + F \tau (A) B + \tau (F) AB - F \tau (I)  AB - F A \tau (I) B . \]
Comparing these equations, we arrive at
\[ F \tau (AB) = F A \tau (B) + F \tau (A) B - F A \tau (I) B  \]
for all $ A , B \in \al $ and $ F \in \alf $. From Lemma~\ref{p2}(ii), it follows that
\[ \tau ( AB ) = A \tau (B) + \tau (A) B - A \tau (I) B \]
for all $ A , B \in \al $ and $ F \in \alf$.
\\ \\
%%-----------------------------------------------------------------------------------------------------------------------
\textbf{Step 5.} For all $ A \in \mathcal{A} $, we have
\[ \tau ( A) - A \tau (I) = \delta (A) - \delta (I) A .\]
\par
From \eqref{eq5} and Step 3, it follows that
\begin{align*}
\tau ( AF ) - AF \tau (I) & = \delta (AF ) - \delta ( I ) A F \\
& = A \delta ( F) + \delta (A) F - A \delta (I) F - \delta (I) A F \\
& = A \tau (F) - AF \tau (I) + \delta (A) F - \delta(I) AF
\end{align*}
for all $ A  \in \al $ and $ F \in \alf $. On the other hand, according to the Step 4, for all $ A \in \al$ and $ F \in \alf $, we have
\[ \tau (AF) - A F \tau (I) = A \tau (F) + \tau ( A) F - A \tau (I) F - A F \tau (I) . \]
By comparing these equations, we find
\[ ( \delta (A) - \delta ( I ) A ) F = ( \tau (A) - A \tau (I) ) F  \]
for all $ A \in \al $ and $ F \in \alf $. From Lemma~\ref{p2}(i), it follows that 
\[  \delta (A) - \delta (I) A= \tau ( A) - A \tau (I) \]
for all $ A \in \al $.
\\ \\
%%-----------------------------------------------------------------------------------------------------------------------
\textbf{Step 6.} For linear map $\gamma: \al \rightarrow \al$ defined by $\gamma(A)=\delta(A)+A\tau(I)=\tau(A)+\delta(I)A$ we have $(\gamma, \delta, \tau)\in Tder(\al)$.
\par
From Steps 3 and 5, it follows that
\begin{align*}
\gamma(AB) & = \delta(AB)+AB\tau(I) \\
& = A \delta ( B) + \delta (A) B - A \delta (I) B+ AB\tau(I)\\
& = \delta (A) B+A(\delta ( B)-\delta (I) B)+ AB\tau(I)\\
& = \delta (A) B+A(\tau ( B)- B\tau ( I))+ AB\tau(I)\\
& = \delta (A) B+A\tau ( B)
\end{align*}
for all $A,B \in \al$. So $(\gamma, \delta, \tau)\in Tder(\al)$.
\\ \\
%%-----------------------------------------------------------------------------------------------------------------------
\textbf{Step 7.} There are elements $R,T\in \al$ such that $\gamma(A)=RA+AT$ for all $A\in \al$.
\par
Define the linear map $ \alpha : \al \to \al $ by $ \alpha(A) = \delta (A) - \delta (I) A $. It follows from Step 3 that 
\begin{align*}
\alpha(AB) & = \delta ( AB ) - \delta (I) AB \\
& = A \delta ( B) + \delta (A) B - A \delta (I ) B - \delta ( I ) AB \\
& = A \alpha (B) + \alpha( A) B .
\end{align*} 
So $ \alpha $ is a derivation. From \cite[Theorem 2.2]{che} every linear derivation of a nest algebra on a Banach space is continuous and by \cite{de0} all continuous linear derivations of a nest algebra on a Banach space are inner derivations. So $\alpha$ is inner, i.e. there exists $ S \in \al $ such that $\alpha( A) = A S - S A $ for all $ A \in \al$. Set $R=\delta(I)-S$ and $T=\tau(I)+S$. So $\delta(A)=RA+AS$ for all $A\in \al$ and
\[\gamma(A)=\delta(A)+A\tau(I)=RA+AS+A\tau(I)=RA+AT \]
for all $A\in \al$, where $R,T\in \al$.
\par
Suppose that $\gamma^{\prime}:\al \rightarrow \al$ is another linear map such that $(\gamma^{\prime}, \delta, \tau)\in Tder(\al)$. By definition of ternary derivation we see that $\gamma(AB)=\gamma^{\prime}(AB)$ for all $A,B\in \al$. Hence $\gamma=\gamma^{\prime}$.
The proof is completed.
\\ \\
\textbf{Proof of Corollary~\ref{inner}:} Assume that $(\gamma, \delta , \tau)\in Tder(\al)$. So $\delta, \tau$ satisfy $(Z)$. By Theorem~\ref{main} there are $R, T\in \al$ such that $\gamma(A)=RA+AT$ for all $A\in \al$. Since $\gamma(AB)= \delta (A) B+A\tau ( B)$ for all $A,B\in \al$, it follows that $R+T=\delta(I)+\tau(I)$. Set $S=\delta(I)-R=T-\tau(I)$. So
\[ \delta(A)=\gamma(A)-A\tau(I)=RA+AT-A\tau(I)=RA+AS\]
and
\[ \tau(A)=\gamma(A)-\delta(I)A=RA+AT-\delta(I)A=-RS+AT\]
for all $A\in \al$. We conclude that $(\gamma, \delta , \tau)$ is an inner ternary derivation.
%%___________________________________________________________________________

\subsection*{Acknowledgment}
The author likes to express his sincere thanks to the referee(s) for this paper.

% ------------------------------------------------------------------------
%GATHER{Xbib.bib}   % For Gather Purpose Only
%GATHER{Paper.bbl}  % For Gather Purpose Only
\bibliographystyle{amsplain}
\bibliography{xbib}

\end{document}